\documentclass[a4paper,11pt,leqno,intlimits]{amsart}

\usepackage{amstext}
\usepackage{amssymb}
\usepackage{graphicx} 
\usepackage{color}
\usepackage{epsfig}
\usepackage{pifont}
\usepackage{fancyhdr}
\usepackage{psfrag}
\usepackage{float}
\usepackage{longtable}
\usepackage{subfigure}
\usepackage{a4wide}
\allowdisplaybreaks[4]

\usepackage[plain,section]{algorithm}
\usepackage{algpseudocode}
\usepackage{framed}

\newenvironment{remark}[1][]{
  \vspace*{3mm}
  \noindent\textbf{Remark.}
} {

  \hfill $\diamond$

}

\newcommand{\distributionSpace}{\mathcal D'}
\newcommand{\detGradPhi}{{\mathcal J}}
\newcommand{\rieszIso}{{\mathcal R}}
\newcommand{\RieszIsoAlt}{\rieszIso^{-1}_V}
\newcommand{\Cstate}{C_{\!\mbox{\tiny state}}}
\newcommand{\Cadjoint}{(\Cstate^4+\Cstate^2)}

\newcommand{\compactVspace}{\mathcal V}

\newcommand{\errorL}{\eps^r + \frac{h^m}{\eps^m}}
\newcommand{\errorH}{\eps^r + \frac{h^m}{\eps^{m+1}}}
\newcommand{\PhiSpace}{P}
\newcommand{\PhiSpaceL}{{P_L}}
\newcommand{\minCont}{$(\mathcal P)\,$}
\newcommand{\minDiscr}{$(\mathcal P_h)\,$}
\newcommand{\locFunc}{\chi_{c}}
\newcommand{\calB}{{\mathcal B}}

\newcommand{\frechet}{Fr\'echet }

\DeclareMathOperator*{\supp}{supp}
\DeclareMathOperator*{\projection}{Proj}

\newtheorem{Definition}{Definition}[section]
\newtheorem{Lemma}[Definition]{Lemma}

\newtheorem{Theorem}[Definition]{Theorem}
\newtheorem{Assumption}[Definition]{Assumption}

\renewcommand{\phi}{\varphi}
\newcommand{\eps}{\varepsilon}
\renewcommand{\epsilon}{\varepsilon}
\renewcommand{\Xi}{\varXi}
\newcommand{\<}{\langle}
\renewcommand{\>}{\rangle}

\renewcommand*{\rho}{\varrho}

\newcommand{\linFuncSpace}{\mathcal L}	

\title{Optimal Control for Burgers Equation using Particle Methods}

\author{Jan Marburger} 
\address{Jan Marburger, Fraunhofer ITWM Kaiserslautern, {\tt{marburger@research.nit-service.de}}}
\author{Ren\'e Pinnau}
\address{Ren\'e Pinnau, TU Kaiserslautern, FB Mathematik, 
{\tt{pinnau@mathematik.uni-kl.de}}}
\date{\today}

\begin{document}

\bibliographystyle{plain} 

\begin{abstract}
 This papers shows the convergence of optimal control problems where the constraint function is discretised by a particle method.
 In particular, we investigate the viscous Burgers equation in the whole space $\mathbb R$ by using distributional particle approximations. The continuous optimisation problem is derived and investigated. Then, the discretisation of the state constraint and the resulting adjoint equation is performed and convergence rates are derived. Moreover, the existence of a converging subsequence of control functions, obtained by the discrete control problem, is shown. Finally, the derived rates are verified numerically.
\end{abstract} 

\maketitle


\vspace{2mm}
\noindent{\bf Subject (MSC).} 49-99, 65K10, 76M28
\vspace{2mm}
 
\section{Introduction}
Finite element and finite volume methods enjoy limited use in deforming
domain and free surface applications due to exorbitant computational demands.
However, particle methods are ideally suitable for simulating those kind of problems \cite{KuhnertSurfTension}.
Over the last thirty years different approaches to these methods were developed.

In classical particle methods, see e.g. \cite{Raviart,Raviart2}, a function is
approximated in a distributional sense, i.e. by Dirac measures as shown in the following
definition.
\begin{Definition}
 \label{def:distrPartApprox}
 A (distributional) particle approximation of a function $\phi:\mathbb R\to\mathbb R$ is
 denoted by
  \[ \phi(x) \simeq \Pi^h \phi := \sum_{i=1}^N h_i \phi(x_i)\delta(x-x_i) \]
 where $h_i$ are quadrature weights.
\end{Definition}
Obviously, the approximation by Dirac measures satisfies
 \[ \int_\Omega \phi(x) \,dx \simeq \int_\Omega \Pi_h \phi \,d\mu(x) 
    = \sum_{i=1}^N h_i \phi(x_i)\]
with an appropriate measure $\mu$.

``Smoothed particle hydrodynamics'' (SPH) \cite{monaghan05, monaghan77, LiLiu}, uses the approximation of the Dirac distribution by
a continuous function with an appropriate choice of the smoothing parameter $\eps$, which is called a Dirac sequence.
Moreover, the particles are equipped by a mass, which allows a good physical interpretation of this method.
This method uses a strong formulation of the equation system.

In contrast, meshless Galerkin and partition of unity methods \cite{schweitzerDipl} use the weak formulation for solving
a system of differential equations. These methods approximate the solution space by basis functions. This basis is
obtained by e.g. RKPM (Reproducing Kernel Particle Method) or MLS (Moving Least Squares) \cite{LiLiu}. Note that this is
a generalisation of finite elements.

Another mesh-less approach is the Finite Pointset Method (FPM). It is similar to the
classical particle method described in the very beginning. The approximation operators,
like gradient or Laplacian, are obtained by finite differences, in particular by applying
a least squares approach.
This method also uses the strong version of a differential equation. For details see
\cite{KuhnertDiss,Kuhnert1, KuhnertSurfTension}.

In this paper we consider the analytical aspects of optimisation using the first
method, i.e. distributional particle approximation. The optimisation, performed by using
the Lagrangian technique, is exemplified on a problem subject to the viscous Burgers equation, i.e.
minimise the cost functional $J(y,u)$ subject to
\begin{align*}
 \begin{aligned}
  \partial_t y + y\partial_x y - \nu \partial_{xx}  y &= u 
   &&\text{in }\mathbb R\times(0,T)\\
   y(x,0) &= y_0(x) &&\text{in }\mathbb R
 \end{aligned}
\end{align*}
where $T>0$ denotes the final time and $\nu>0$ the viscosity.

First results and strategies on
applying optimisation to particle methods are shown in \cite{JM_NM_RP_Basics,JMpamm09}.
The application to optimisation of free surface problems is investigated in \cite{JMparticles09,JMpamm10}.

We start with deriving the adjoint equation and the gradient of the continuous problem. Moreover, we prove existence and uniqueness of the adjoint and show the existence of a solution to the minimisation problem.
Then, the state and adjoint equation is discretised by a particle method. Again, we investigate the resulting system and show the existence of a optimal control for the discrete optimisation problem.
Finally, we show the existence of a converging subsequence of controls obtained by the discrete optimiation to the analytical solution and confirm the derived results numerically.

Throughout this paper, we use the following notation. The sets $V:=H^1(\mathbb R)$,
$H:=L^2(\mathbb R)$, $L^2(V) := L^2(0,T;V)$ and 
$W(V) := \{ y\in L^2(V) \,|\, \partial_t y \in L^2(V^*) \}$ are Hilbert-spaces.
For these spaces the embedding $V \hookrightarrow H = H^* \hookrightarrow V^*$ holds.
The space $\compactVspace\subset V$ denotes the space of functions with bounded support, i.e. $y\in \compactVspace \Rightarrow \overline{\supp y} \subseteq [a,b]$, $\infty<a\le b<\infty$. Moreover, we denote the Riesz isomorphism by $\rieszIso_V:V\to V^*$.

Among others, we use the following relations.
The estimation
\begin{align*}
 2ab \le \eps a^2 + \eps^{-1} b^2
 \qquad\mbox{for } a,b\in \mathbb R,\;\eps > 0
\end{align*}
is called \emph{Young's inequality} \cite{Alt}. To estimate the $V^*$ norm, we use
\begin{align*}
 \|\phi\|_{V^*}^2 = \<\RieszIsoAlt \phi, \RieszIsoAlt \phi\>_{V}
  = \<\rieszIso_V\RieszIsoAlt \phi, \RieszIsoAlt \phi\>_{V^*,V}
  = \<\phi, \RieszIsoAlt \phi\>_{V^*,V}
\end{align*}
for $\phi\in V^*$, which is a direct consequence of the Riesz representation theorem.
For the estimation of the time dependency, we state \emph{Gronwall's lemma} which reads
for $u:[0,T]\to\mathbb R$
\begin{align*}
 u(T) \le c_1 + \int_0^T g(\tau) u(\tau)\,d\tau 
 \qquad\Rightarrow\qquad
 u(T) \le c_1 \exp \Big( \int_0^T g(\tau)\,d\tau \Big)
\end{align*}
where $c_1\in\mathbb R$ is a constant and $g:[0,T]\to\mathbb R_0^+$. 

\nocite{casas06,volkwein-basics, Troel}

\section{Optimal Control Problem}
\label{sec:contOptimisation}
We define the set of admissible control functions by
\[ U_{ad} := \{ u\in H^1(0,T) \,|\, u_l \le u(t) \le u_u \;a.e. \} \]
for a fixed final time $T>0$ and $-\infty < u_l \le u_u <\infty$.
We consider the following minimisation problem: Minimise
\begin{align*}
 J(y,u) := \frac12 \| y(T)-y_d \|_{H}^2 + \frac\sigma2 \| u \|_{H^1(0,T)}^2
\end{align*}
over $(y,u) \in W(V)\times U_{ad}$
subject to the viscous Burgers equation
\begin{align}
\label{equ:state}
 \begin{aligned}
  \partial_t y + y\partial_x y - \nu \partial_{xx}  y &= \locFunc u 
   &&\text{in }\mathbb R\times(0,T)\\
   y(x,0) &= y_0(x) &&\text{in }\mathbb R
 \end{aligned}
\end{align}
where $\nu>0$ denotes the viscosity, $\locFunc\in C^\infty_0(\mathbb R)$ a spatial 
localisation function with bounded support and $y_d\in H$ the desired state at final time $T$. Both functions 
are supposed to have compact support.

The weak formulation of \eqref{equ:state} is given by
\begin{align}
 \label{equ:weakState}
 \begin{aligned}
 \frac{d}{dt}\<y(t),\phi\>_H + \<y(t)\partial_x y(t),\phi\>_H
  + \nu\<\partial_x y(t),\partial_x \phi\>_H &= \<\calB u(t),\phi\>_H \\
 \<y(0),\psi\>_H &= \<y_0,\psi\>_H
 \end{aligned}
\end{align}
for all $\phi\in V, \psi\in H$ and a.e. $t\in(0,T)$. 
Here $\calB:H^1(0,T)\to L^2(V^*)$ is defined by
$(\calB u)(x,t) = \locFunc(x) u(t)$.
Hence, the following relation holds
\begin{align*}
 \<\calB u, \phi\>_{L^2(V^*),L^2(V)} = \int_0^T u(t) \int_\mathbb R \locFunc(x) \phi(x,t) \,dx\,dt
 \le C \|u\|_{L^2(0,T)} \|\phi\|_{L^2(H)}
\end{align*}
where $C=\|\chi_c\|_{L^\infty}$.
In the following we state the existence of a unique solution to \eqref{equ:weakState} and show its boundedness.
\begin{Theorem}
 \label{thm:existenceForward}
 Let $y_0\in \compactVspace$ and $u\in H^1(0,T)$. Then there exists a unique weak solution $y\in W(V)\cap L^\infty(0,T;H^2(\mathbb R))$
 to the viscous Burgers equation \eqref{equ:weakState}. Moreover, $\partial_t y \in L^2(0,T;H^1(\mathbb R))$.
\end{Theorem}
\begin{proof}
 For the existence and uniqueness we refer to \cite{existBurger2}.
 As the initial value $y_0$ has bounded support and the right hand side is sufficiently smooth, in particular $\chi\in C^\infty_0(\mathbb R)$ and $u\in C^0(0,T)$, the proof in \cite{existBurger2} is also valid for \eqref{equ:weakState}.
\end{proof}

\begin{Theorem}
 \label{thm:boundedForward}
 Let $u\in H^1(0,T)$, $y_0\in \compactVspace$ and $y\in W(V)\cap L^\infty(V)$ be the solution to the viscous Burgers equation \eqref{equ:weakState}.
 There exists a constant $C>0$, only depending on $\nu$, $T$, and $\locFunc$, such that
 \begin{align*}
  \|y\|_{L^\infty(V)} + \|y\|_{W(V)}\le  C( \|u\|_{L^2(0,T)} + \|y_0\|_V )^2 \exp( C(\|u\|_{L^2(0,T)}^2 + \|y_0\|_H^2) ) =: \Cstate
 \end{align*}
 holds.
\end{Theorem}
\begin{proof}
 For $y(t)\in V$ 
 \begin{align*}
  \int_{\mathbb R} y(t)\partial_x y(t) y(t)\,dx = - 2\int_{\mathbb R} y(t)\partial_x y(t)y(t) \,dx
  \Rightarrow \<y(t)\partial_x y(t),y(t)\>_H = 0
 \end{align*}
 holds.
 Hence we get
 \begin{align*}
  \frac{1}{2} \frac{d}{dt} \|y(t)\|_{H}^2 + \nu \|y(t)\|_{V}^2
   &= \<\calB u(t),y(t)\>_{V^*,V} + \nu\|y(t)\|_H^2 \\
  &\le C|u(t)|\|y(t)\|_{V} + \nu\|y(t)\|_H^2 \\
  &\le \gamma^{-1}C|u(t)|^2 + \gamma C \|y(t)\|_V^2 + \nu\|y(t)\|_H^2
 \end{align*}
 by multiplying the state equation (\ref{equ:discrState}) by $y$
 and integration over $\mathbb R$.
 Setting $\gamma=\frac{\nu}{2C}$ gives
 \begin{align}
  \label{equ:ddtyEst}
  \frac{d}{dt}\|y(t)\|_H^2 + \nu \|y(t)\|_V^2 \le \frac{4C^2}{\nu}|u(t)|^2 + 2\nu \|y(t)\|_H^2
 \end{align}
 Integration over $(0,t)$ yields
 \begin{align*}
  \|y(t)\|_H^2 - \|y(0)\|_H^2 + \nu\|y\|_{L^2(0,t;V)}^2 
   \le \frac{4C^2}{\nu} \|u\|_{L^2(0,t)}^2 + 2\nu\int_0^t \|y(\tau)\|_H^2 \,d\tau
 \end{align*}
 Due to Gronwall's lemma we obtain
 \begin{align}
  \label{equ:yLinfHest}
  \|y\|_{L^\infty(H)}^2 \le C(\|u\|_{L^2(0,T)}^2 + \|y_0\|_H^2)
 \end{align}
 and
 \begin{align}
  \label{equ:yL2Vest}
  \|y\|_{L^2(V)}^2 \le C(\|u\|_{L^2(0,T)}^2 + \|y_0\|_H^2).
 \end{align}
 The $L^2(V^*)$ estimation is obtained by multiplying $\RieszIsoAlt\partial_t y$ and
 integrating over $\mathbb R\times(0,T)$
 \begin{align*}
  \|\partial_t y\|_{L^2(V^*)}^2 &= 
   \<\partial_t y,\RieszIsoAlt\partial_t y\>_{L^2(V^*),L^2(V)} \\
  &= -\<y\partial_x y,\RieszIsoAlt\partial_t y\>_{L^2(H)}
  -\nu\<\partial_x y, \partial_x\RieszIsoAlt\partial_t y \>_{L^2(H)} 
  + \<\calB u,\RieszIsoAlt\partial_t y\>_{L^2(V^*),L^2(V)}  \\
  &\le \|y\|_{L^\infty(H)}\|y\|_{L^2(V)}\|\RieszIsoAlt\partial_t y\|_{L^2(L^\infty)} + \big(\nu\|y\|_{L^2(V)} + C\|u\|_{L^2(0,T)}  \big)
   \|\partial_t y\|_{L^2(V^*)} \\
  &\le \big( \|y\|_{L^\infty(H)}\|y\|_{L^2(V)} + \nu\|y\|_{L^2(V)} + C\|u\|_{L^2(0,T)}  \big)
   \|\partial_t y\|_{L^2(V^*)}
 \end{align*}
 due to the embedding $V\hookrightarrow H$ and $V\hookrightarrow L^\infty$.
 Since \eqref{equ:yL2Vest} and \eqref{equ:yLinfHest} hold we obtain
 \begin{align*}
  \|\partial_t y\|_{L^2(V^*)} \le C \big( \|u\|_{L^2(0,T)} + \|y_0\|_H \big)^2
 \end{align*}
 which yields
 \begin{align*}
  \|y\|_{W(V)} \le C \big( \|u\|_{L^2(0,T)} + \|y_0\|_H \big)^2
 \end{align*}
 Finally, we estimate the $L^\infty(V)$ bound by multiplying the state equation (\ref{equ:discrState}) by $-\partial_{xx} y$ and integrating over $\mathbb R$.
 \begin{align*}
  \frac12 \frac{d}{dt} \|\partial_x y(t)\|_H^2 &+ \nu \|y(t)\|_{H^2}^2 = \<\locFunc' u(t),\partial_x y(t)\>_H + \<y(t)\partial_x y(t),\partial_{xx} y(t)\>_H
   + \nu\|y(t)\|_{V}^2 \\
  &\le C|u(t)| \|\partial_x y(t)\|_H + \|y(t)\|_{L^\infty}\|\partial_x y(t)\|_H\|\partial_{xx} y(t)\|_H + \nu \|y(t)\|_V^2 \\
  &\le C|u(t)| \|\partial_x y(t)\|_H + (2\gamma)^{-1} \|y(t)\|_{L^\infty}^2 \|\partial_x y(t)\|_H^2
   + \frac{\gamma}{2}\|y(t)\|_{H^2}^2 + \nu \|y(t)\|_V^2
 \end{align*}
 Setting $\gamma=\nu$ and using Young's inequality yields
 \begin{align*}
  \frac{d}{dt}\|\partial_x y(t)\|_H^2 &+ \frac{\nu}{2} \|y(t)\|_{H^2}^2 \le \frac{C^2}{2} |u(t)|^2 + \frac{1}{2\nu}\|\partial_x y(t)\|_H^2 
   + \frac2\nu \|y(t)\|_{L^\infty}^2 \|\partial_x y(t)\|_H^2 + \nu \|y(t)\|_V^2
 \end{align*}
 Integrating over $(0,T)$ yields
 \begin{align*}
  \|\partial_x y(T)\|_H^2 - \|\partial_x y(0)\|_H^2 + \nu\|y\|_{L^2(H^2)}^2
  &\le C \|u\|_{L^2(0,T)}^2 + \nu \|y\|_{L^2(V)}^2 \\
  &\quad+ \int_0^T (1+\frac{2}{\nu}) \|y(\tau)\|_{L^\infty}^2) \|\partial_x y(\tau)\|_H^2 \,d\tau 
 \end{align*}
 By applying Gronwall's lemma we obtain
 \begin{align*}
  \|y\|_{L^2(H^2)}^2 &\le C( \|u\|_{L^2(0,T)}^2 + \|y_0\|_V^2 ) \exp(c T \int_0^T \|y(\tau)\|_{L^\infty}^2\,d\tau ) \\
  & \le C( \|u\|_{L^2(0,T)}^2 + \|y_0\|_V^2 ) \exp(C( \|u\|_{L^2(0,T)}^2 + \|y_0\|_H^2) )
 \end{align*}
 as $V\hookrightarrow L^\infty$ and hence $\|y\|_{L^2(L^\infty)} \le c\|y\|_{L^2(V)} \le C(\|u\|_{L^2(0,T)}^2 + \|y_0\|_H^2 )$ due to \eqref{equ:yL2Vest}.
 Analogous to the derivation of \eqref{equ:yLinfHest} we obtain
 \begin{align*}
  \|y\|_{L^\infty(V)}^2 \le C( \|u\|_{L^2(0,T)}^2 + \|y_0\|_V^2 ) \exp( C(\|u\|_{L^2(0,T)}^2 + \|y_0\|_H^2) )
 \end{align*}
 Combining all results yields the assumption.
 
\end{proof}

We introduce the constraint function $e=(e_1,e_2):X\to Z^*$ where $X=W(V) \times U_{ad}$ 
and $Z=L^2(V)\times H$. By defining $e$ as
\begin{align}
 \label{equ:weakStateE}
 \begin{aligned}
  e(y,u) :=
  \begin{pmatrix}
   \partial_t y + y\partial_x y - \nu\partial_{xx} y - \calB u \\
   y(0) - y_0
  \end{pmatrix}
 \end{aligned}  
\end{align}
the optimal control problem can be understood as the constrained minimisation problem:
\begin{flalign*}
 &\mbox{\minCont} \hspace{40mm}
 \mbox{minimise }J(y,u)
 \qquad\mbox{s.t.}\qquad
 e(y,u) = 0&
\end{flalign*}
over $(y,u)\in X$.
The \frechet derivatives with
respect to $x=(y,u)$ are denoted by a prime and with respect to $y$ and $u$ by $e_y$ and
$e_u$, respectively.

\begin{Theorem}
 The cost functional $J$ and the constraint function $e$ are twice Fr\'echet differentiable and their second Fr\'echet
 derivatives are Lipschitz-continuous on $V\times H^1(0,T)$.
\end{Theorem}
\begin{proof}
 See \cite{MyPhd}, p. 66.
\end{proof}

\noindent The following theorem states the existence of a solution to the optimal control
 problem \minCont.
\begin{Theorem}
 \label{thm:existOptimalSolution}
 There exists an optimal solution to \minCont.
\end{Theorem}
\begin{proof}
 Due to theorem \ref{thm:existenceForward} there exists a unique solution $y\in W(V)\cap L^\infty(V)$
 to \eqref{equ:weakState} for every  $u\in U_{ad}$.
 Moreover, $J(y,u)\ge 0$ for all $y\in W(V)$ and $u\in U_{ad}$. Hence,
 \begin{align*}
  j = \inf_{u\in U_{ad}} J(y(u),u) \in \mathbb R^+_0
 \end{align*}
 exists.
 We define a minimising sequence $(y_n,u_n)\in X$ by
 \begin{align*}
  e(y_n,u_n) = 0 \qquad\mbox{ and }\qquad J(y_n,u_n) \to j \quad\mbox{as}\quad n\to\infty
 \end{align*}
 Since $U_{ad}\subset L^\infty(0,T)$ is bounded there exists a subsequence of $u_n$, 
 again denoted by $u_n$, with
 \begin{align}
  \label{equ:existOpt1}
  u_n \rightharpoonup^* \bar u,\quad n\to\infty\qquad \mbox{in }L^\infty(0,T) 
 \end{align}
  and $\bar u\in U_{ad}$ as $U_{ad}$ is weakly closed.
 The states $y_n$ are bounded $\|y_n\|_{W(V)}\le C$. Hence, there exists a subsequence 
 of $y_n$ with
 \begin{align}
  \label{equ:existOpt2}
  y_n \rightharpoonup \bar y,\quad n\to\infty\qquad \mbox{in }W(V)
 \end{align}
 $y_n$ is also bounded in $L^\infty(V)$, i.e. $\|y_n\|_{L^\infty(V)}\le C$, and hence we obtain a subsequence
 \begin{align*}
  y_n(t) \rightharpoonup \bar y(t),\quad n\to\infty\qquad \mbox{in } V
 \end{align*}
 for a.e. $t\in (0,T)$ and due to the embedding $V\hookrightarrow C^{0,\alpha}(\mathbb R)$, for $0<\alpha<\frac12$, we get
 \begin{align*}
    y_n(t) \rightharpoonup^* \bar y(t),\quad n\to\infty\qquad \mbox{in } L^\infty(\mathbb R)
 \end{align*}
 due to Banach Alaoglu \cite{Alt}. Furthermore, this implies
 \begin{align}
  \label{equ:existOpt3}
    y_n(t)^2 \rightharpoonup^* \bar y(t)^2,\quad n\to\infty\qquad \mbox{in } L^\infty(\mathbb R) .
 \end{align}
 Since \eqref{equ:existOpt2} holds we obtain for all $\phi\in L^2(V)$
 \begin{align*}
  \int_0^T \<\partial_t y_n(t),\phi(t)\>_{V^*,V} + \<y_n,\phi(t)\>_V \,dt 
  \to \int_0^T \<\partial_t \bar y(t),\phi(t)\>_{V^*,V} + \<\bar y,\phi(t)\>_V \,dt
 \end{align*}
 and
 \begin{align*}
  -2 \int_0^T \<y_n(t)\partial_x y_n(t),\phi(t)\>_H \,dt
  = \int_0^T \<y_n(t)^2,\partial_x\phi(t)\>_H \,dt
  \to \int_0^T \<\bar y(t)^2,\partial_x\phi(t)\>_H \,dt
 \end{align*}
 due to \eqref{equ:existOpt3}.
 Moreover, we get
 \begin{align*}
  \<\calB u_n,\phi\>_{L^2(V^*),L^2(V)} \to \<\calB \bar u,\phi\>_{L^2(V^*),L^2(V)}
 \end{align*}
 as 
 \begin{align*}
  \<\calB u_n,\phi\>_{L^2(V^*),L^2(V)} 
  = \<u_n, 
  \underbrace{\int\nolimits_{\mathbb R} \locFunc(x)\phi(x,t)\,dx}_{\in L^2(0,T)}\>_{L^2(0,T)}
 \end{align*}
 and \eqref{equ:existOpt1} hold.
 This yields $e_1(\bar y,\bar u) = 0$.
 The convergence of $y_n$ to $\bar y$ in $W(V)$ also yields 
 $y_n(0)\rightharpoonup\bar y(0)$ in $H$ and hence
 \begin{align*}
  \< y_n(0),\psi\>_H \to \< \bar y(0),\psi\>_H
 \end{align*}
 which yields $e_2(\bar y,\bar u)=0$.
 Finally, we conclude
 \begin{align*}
  e(\bar y,\bar u) = (e_1(\bar y,\bar u), e_2(\bar y,\bar u)) = 0
 \end{align*}
 Due to definition $J$ is lower-semi-continuous, i.e.
 \begin{align*}
  J(\bar y,\bar u) \le \liminf_{n\to\infty} J(y_n,u_n) = j
 \end{align*}
 As $J(\bar y,\bar u)\le j$ we obtain $J(\bar y,\bar u) = j$ and hence $j$ is a minimum.
 
\end{proof}
We define the Lagrange functional by
\begin{align*}
 L(y,u,p,\lambda) := J(y,u) + \<e(y,u),(p,\lambda)\>_{Z^*,Z}
\end{align*}
which enables us to state the first order optimality condition.
\begin{Theorem}
 \label{thm:existLagMult}
 Let $(\bar y,\bar u)\in W(V)\cap L^\infty(V)\times U_{ad}$ be an optimal solution to \minCont.
 There exist Lagrange multipliers $\bar p\in W(V)$ and $\bar\lambda\in H$ satisfying the
 first order necessary optimality condition
 \begin{align}
  \label{equ:firstOrderOptCont}
  L_y(\bar y,\bar u, \bar p,\bar\lambda)  = 0,
  \qquad
  L_u(\bar y,\bar u, \bar p,\bar\lambda)(u-\bar u) \ge 0 \quad\forall u\in U_{ad}
  \qquad\mbox{ and }\qquad
  e(\bar y,\bar u) = 0.
 \end{align}
\end{Theorem}
\begin{proof}
 For all $u\in U_{ad}$ there exists a unique $y\in W(V)$ such that
 $e(y,u) = 0$.
 If $e_y(y(u),u)\in\linFuncSpace(W(V),Z^*)$ has a bounded inverse for all $u\in U_{ad}$
 then there exist Lagrange multiplier $(\bar p,\bar\lambda)\in Z$ satisfying
 \eqref{equ:firstOrderOptCont}, cf. e.g. \cite{Pinnau}.
 Therefore, we show the bijectivity of $e_y(y,u)$, that is, for all 
 $(\eta,\xi)\in Z$ there exists a $v\in L^2(V)$ such that
 \begin{align*}
  e_y(y,u)v = ( D_y e_1(y,u), D_y e_2(y,u) ) = (\eta,\xi)
 \end{align*}
 Due to the definition of $e$ we get for the derivative with respect to the state
 \begin{align*}
  \partial_t v + \partial_x(yv) - \nu \partial_{xx} v &= \rieszIso_V\eta
  &&\mbox{in }L^2(V^*) \\
  v(0) &= \xi
  &&\mbox{in }H \\
 \end{align*}
 Introducing the bilinear form $a(t;\cdot,\cdot):V\times V \to \mathbb R$ by
 \begin{align*}
  a(t;w,q) := \int_{\mathbb R} \nu \partial_x w\partial_x q - y(t)q\partial_x w\,dx
 \end{align*}
 we rewrite $e_y(y,u)$ as
 \begin{align}
  \label{equ:weakFormEy}
  \frac{d}{dt}\<v(t),\psi\>_H + a(t;v(t),\psi) = \<\rieszIso_V\eta(t),\psi\>_{V^*,V}
 \end{align}
 for all $\psi\in V$ and a.e. $t\in (0,T)$.
 To show the unique solvability of \eqref{equ:weakFormEy} it suffices to show that
 $a(t;\cdot,\cdot)$ is continuous and weak $V$-coercive, i.e.
 \begin{align*}
  \exists\kappa_V>0,\kappa_H\ge 0 \;:\;
  a(t;w,w) \ge \kappa_V \|w\|_V^2 - \kappa_H \|w\|_H^2
  \qquad\mbox{for all }w\in V.
 \end{align*}
 (cf. e.g. \cite{ShowalterII}, p. 112).
 The bilinear form $a(t;\cdot,\cdot)$ is continuous as
 \begin{align*}
  |a(t;w,q)| &= \big|-\<y(t)\partial_x w,q\>_H + \nu\<\partial_x w,\partial_x q\>_H\big|
   \le C\|y\|_{L^\infty} \|w\|_V\|q\|_H + \nu \|w\|_V\|q\|_V \\
  &\le C \|w\|_V \|q\|_V
 \end{align*}
 The $V$-coercivity is derived by
 \begin{align*}
  a(t;w,w) &= -\<y(t)\partial_x w,w\>_H + \nu\<\partial_x w,\partial_x w\>_H
   = -\<y(t)\partial_x w,w\>_H + \nu\|w\|_V^2 - \nu\|w\|_H^2\\
  & \ge -\|y(t)\|_{L^\infty}\|w\|_V\|w\|_H + \nu\|w\|_V^2 - \nu\|w\|_H^2 \\
  & \ge -\gamma C \|w\|_V^2 - \gamma^{-1}C \|w\|_H^2 + \nu \|w\|_V^2 - \nu\|w\|_H^2
  \intertext{due to Young's inequality and theorem \ref{thm:boundedForward}}
  & \ge \frac12 \nu \|w\|_V^2 - c \|w\|_H^2
 \end{align*}
 with $\gamma:=\frac{\nu}{2C}$.
 Hence, $e_y(y,u)$ is bijective.

\end{proof}

\noindent The adjoint equation is given by
\begin{align}
 \label{equ:adjoint}
 \begin{aligned}
  \frac{d}{dt} \<p(t),\phi\>_H + \<y(t)\partial_x p(t), \phi\>_H - \nu\<\partial_x p(t),\partial_x\phi\>_H &= 0 \\
  \< p(T), \psi\>_H &= \< y_d - y(T), \psi\>_H
 \end{aligned}
\end{align}
for all $\phi\in V$ and $\psi\in H$.
The variational inequality yields with $\hat J(u)=J(y(u),u)$ for a minimum $\bar u\in U_{ad}$
\begin{align*}
 \<D_u\hat J(\bar u),u-\bar u\>_{H^1(0,T)^*,H^1(0,T)} = \sigma \<u,u-\bar u\>_{H^1(0,T)}
  - \<\calB^* p,u-\bar u\>_{H^1(0,T)^*,H^1(0,T)}
\end{align*}
and hence the gradient reads
\begin{align*}
 \nabla \hat J(u) := \rieszIso^{-1}_{H^1(0,T)} D_u L
\end{align*}
A detailed derivation of the adjoint equation and gradient can be found in \cite{BurgersVolk}.

\begin{Theorem}
 \label{thm:boundedAdjointCont}
 The adjoint equation has a unique solution $p\in W(V)$. Moreover, there exists a constant $C>0$ such that
 \begin{align*}
  \|p\|_{W(V)} \le C \|y(T)-y_d\|_H
 \end{align*}
 holds.
\end{Theorem}
\begin{proof}
 The existence and uniqueness is show in the proof of theorem \ref{thm:existLagMult}. 
 The estimation is a direct consequence of the existence proof, see e.g. \cite{ShowalterII}, p. 112 or \cite{ZeidlerIIA}, p. 424.
 
\end{proof}

\begin{Lemma}
 \label{lem:lipschitzAdjoint}
 Let $p(y)\in W(V)$ be the solution of (\ref{equ:adjoint}) for $y\in W(V)\cap L^\infty(V)$.
 Then
  \[ \| p(y_1) - p(y_2) \|_{L^2(V)} \le C\| y_1 - y_2 \|_{W(V)} \]
 holds for $y_1,y_2\in W(V)$.
\end{Lemma}
\begin{proof}
 We define $p_i := p(y_i)$.
 The adjoint equation \eqref{equ:adjoint} yields for $\tilde p_i(t) = p_i(T-t)$
 and $\tilde y_i(t)=y_i(T-t)$
 \begin{align*}
  \<\partial_t \tilde p_i(t), \phi\>_{V^*,V} 
   - \<\tilde y_i(t)\partial_x \tilde p_i(t),\phi\>_H
   + \nu \<\partial_x \tilde p_i(t),\partial_x \phi\>_H = 0
 \end{align*}
 for all $\phi\in V$. Introducing $z(t) := \tilde p_1(t)-\tilde p_2(t)$ 
 and setting $\phi = z(t)$ yields
 \begin{align*}
  \<\partial_t z(t),z(t)\>_{V^*,V} 
   - \<(\tilde y_1(t)-\tilde y_2(t))\partial_x \tilde p_1(t) 
   - \tilde y_2(t)\partial_x z(t),z(t)\>_H + \<\partial_x z(t),\partial_x z(t)\>_H = 0
 \end{align*}
 Therefore, we obtain
 \begin{align*}
  \frac{1}{2}\frac{d}{dt}\|z(t)&\|_H^2 + \nu\|z(t)\|_V^2 = 
   \<(\tilde y_1(t)-\tilde y_2(t))\partial_x \tilde p_1(t),z(t)\>_H
   - \<\tilde y_2(t)\partial_x z(t),z(t)\>_H + \nu \|z(t)\|_H^2 \\
  &\le c\|\partial_x\tilde p_1(t)\|_{H} \|\tilde y_1(t)-\tilde y_2(t)\|_{L^\infty} \|z(t)\|_H 
   + c\|\tilde y_2(t)\|_{L^\infty}\|z(t)\|_V\|z(t)\|_H + \nu \|z(t)\|_H^2 \\
  &\le C\|\tilde p_1(t)\|_{V} \|\tilde y_1(t)-\tilde y_2(t)\|_{V} \|z(t)\|_H
   + C\|z(t)\|_V\|z(t)\|_H + \nu \|z(t)\|_H^2 \\
  &\le C\|\tilde y_1(t)-\tilde y_2(t)\|_V^2 + C\|\tilde p_1(t)\|_V^2\|z(t)\|_H^2 
   + C\gamma \|z(t)\|_V^2 + (C\gamma^{-1}+\nu) \|z(t)\|_H^2 \\
  &\le C \|\tilde y_1(t)-\tilde y_2(t)\|_V^2 + (\frac{2C^2}{\nu} +\nu)\|\tilde p_1(t)\|_V^2\|z(t)\|_H^2
   + \frac12 \nu \|z(t)\|_V^2
 \end{align*}
 by setting $\gamma = \frac{\nu}{2C}$.
 Integrating over $(0,T)$ yields
 \begin{align*}
  \|z(T)\|_H^2 + \nu\|z\|_{L^2(V)}^2
   \le C \|\tilde y_1-\tilde y_2\|_{L^2(V)}^2 + \|z(0)\|_H^2 + C\int_0^T \|\tilde p_1(t)\|_V^2 \|z(t)\|_H^2 \,dt
 \end{align*}
 Applying Gronwall's lemma we obtain
 \begin{align*}
  \nu\|z\|_{L^2(V)}^2 \le C (\|\tilde y_1-\tilde y_2\|_{L^2(V)}^2 + \|z(0)\|_H^2)
   \exp\Big( C \underbrace{\int_0^T \|\tilde p_1(t)\|_V^2 \,dt}_{=\|\tilde p_1\|_{L^2(V)}^2 \le C} \Big) + \|z(0)\|_H^2
 \end{align*}
 which finally yields
 \begin{align*}
  \|p_1-p_2\|_{L^2(V)}^2 \le C \|y_1-y_2\|_{L^2(V)}^2
   + C \|y_1(T)-y_2(T)\|_{H}^2
 \end{align*}
 The second term of the right hand side is estimated by
 \begin{align*}
  \| y_1(T)-y_2(T) \|_H^2 \le C =  C \|y_1-y_2\|_{W(V)}
 \end{align*}
 due to the embedding $W(V)\hookrightarrow C^0(H)$.
 Combining the results yields
  \[ \| p_1 - p_2 \|_{L^2(V)} \le C\| y_1 - y_2 \|_{W(V)} \] 
\end{proof}

\section{Discretisation via Particle Methods}
\label{sec:particleDiscr}
In this paper we consider the classical particle method. 
This approach approximates an arbitrary function $y\in C^k(\mathbb R)$,
by a finite dimensional basis of Dirac delta distributions. In particular we obtain for $k\ge 0$
the approximation operator $\Pi^h: C^k(\mathbb R) \to \distributionSpace(\mathbb R)$
\begin{align*}
 \Pi^h y(x) := \sum_{i=1}^N y(x_i) \delta(x-x_i) \omega_i
\end{align*}
where $\delta$ denotes the Dirac delta distribution, $x_i$ are the supporting points and
$\omega_i$ are quadrature weights, cf. \cite{Raviart}. This approximation is motivated by
\begin{align*}
 \<y,\phi\> = \int_\Omega y(x)\phi(x)\,dx 
  \simeq \sum_{i=1}^N y(x_i)\phi(x_i)\omega_i 
  = \< \Pi^h y(x), \phi \>
\end{align*}
for appropriate functions $y,\phi$ and inner products.

\begin{remark}
In case of time dependent interpolations the supporting points are moving.
Let $v\in L^\infty(\mathbb R \times(0,T))$ be a given velocity field. Then
the time dependent supporting points are given by the characteristic curve
\[
 \partial_t \Phi(X,t) = v(\Phi(X,t),t)
 \qquad\mbox{and}\qquad
 \Phi(X,0) = X
\]
and the time dependent interpolation operator
\begin{align*}
  \Pi^h(t) y(x) := \sum_{i=1}^N y(\Phi(X_i,t)) \delta(x-\Phi(X_i,t)) \omega_i(t)
\end{align*}
Note that the quadrature weights $\omega_i$ are time-dependent, in particular they are
depending on the
positions $\Phi(X_i,t)$.
These weights can be obtained by $\omega_i(t) = \detGradPhi(X_i,t)\omega_i(0)$, where $\detGradPhi :=\det(\nabla\Phi)$, since
\begin{align*}
 \int_{\Phi(\Omega,t)} y(x,t) \, dx 
  = \int_\Omega y(\Phi(X,t),t) \det(\nabla\Phi(X,T)) \,dX
\end{align*}
and $\Phi(\Omega,0) = \Omega$.
\end{remark}

In order to obtain a continuous approximation of $y$ it is possible to
``smooth'' the Dirac delta distribution by using a Dirac sequence, i.e.
convolve the Dirac delta distribution with a smoothing kernel, cf. e.g. \cite{Raviart}.
Hence, we define the continuous approximation operator $\Pi^h_\eps: C^k(\mathbb R) \to C^\infty(\mathbb R)$ by
\begin{align*}
 y_h:= \Pi^h_\eps y(x) := \sum_{i=1}^N y(x_i) \delta_\eps(x-x_i) \omega_i
\end{align*}
where $\delta_\eps$ denotes a Dirac sequence as defined in the following lemma.
\begin{Lemma}
\label{lem:convolutionEstimate}
 Assume that there exists an integer $r \ge 1$ such that
 \begin{align}
  \label{equ:regularityDeltaEps}
  \left\{\quad \begin{aligned}
   &\int_{\mathbb R^d} \zeta(x) \,dx = 1\\
   &\int_{\mathbb R^d} x^\alpha \zeta(x)\,dx = 0 
     \quad\forall \alpha\in\mathbb N^d \quad \mbox{with}\quad 1\le\alpha\le r-1 \\
   &\int_{\mathbb R^d} |x|^r |\zeta(x)| \,dx < \infty
  \end{aligned} \right.
 \end{align}
 for $\zeta\in C^0(\mathbb R^d)\cap L^1(\mathbb R^d)$.
 Moreover, $\delta_\eps = \eps^{-d} \zeta(\eps^{-1} x)$.
 Then we have for some constant $c>0$ and for all functions $y\in W^{r,p}(\mathbb R^d)$, $1\le p\le \infty$
 \begin{align*}
  \| y * \delta_\eps - y\|_{L^p} \le c \eps^r |y|_{W^{r,p}} 
 \end{align*}
\end{Lemma}
\begin{proof}
 See \cite{Raviart}, p. 267.
\end{proof}
\noindent In the following we only use smooth kernel functions $\zeta \in C^\infty(\mathbb R)$.

\noindent The handling of time-dependency is analogous to the previous one, in particular
\begin{align*}
 y_h(x,t):= \Pi^h_\eps(t) y(x) := \sum_{i=1}^N y(\Phi(X_i,t)) \delta_\eps(x-\Phi(X_i,t))
\omega_i(t)
\end{align*}
The interpolation error for the smooth operator $\Pi_\eps^h(t)$ is stated in the following
theorem.

\begin{Theorem}
 \label{thm:interpolationError}
 Let the velocity $v\in L^\infty((0,T);W^{m+1,\infty}(\mathbb R))$ 
 and $\delta_\eps\in W^{m+s,1}(\mathbb R)$ be as stated in lemma \ref{lem:convolutionEstimate} for $0<\eps\le1$.
 Then there exists some constant $c>0$ such that for all $y\in W^{k,p}(\mathbb R)$,
 $k=\max(m+s,r)$ and $1<p<\infty$
 \[
   \| y - \Pi^h_\eps(t) y \|_{W^{s,p}} \le c \big( \eps^r |y|_{W^{r,p}} 
    + \frac{h^m}{\eps^{m+s}} \|y\|_{W^{m,p}} \big)
 \]
 holds. 
\end{Theorem}
\begin{proof}
 See \cite{Raviart2} for the estimation of the seminorm or
 \cite{MyPhd} for the above estimate of the norm.
 
\end{proof}
\noindent More details about the analytical background of this method can be found in
\cite{Raviart,Raviart2,MyPhd}.

\section{Discretisation of the Optimal Control Problem}
\label{sec:DiscrOptimisation}
In this section we state the discretisation of the forward problem by a
particle method and the corresponding optimal control problem.
We derive the discretisation error of the forward and adjoint system and estimate the discrepancy between the optimal control function obtained by the analytical approach and the one obtained by the particle approach.

First we discretise the forward system \eqref{equ:state} by the method described in the previous section. 
For this, we introduce the spaces $\PhiSpace := (H^1(0,T))^N$ and $\PhiSpaceL := (L^2(0,T))^N$ with the following inner products and norms
\begin{align*}
 \< \xi,\eta\>_\PhiSpaceL &:= \sum_{i=1}^N h_i \int_0^T \xi_i(t)\eta_i(t)\,dt 
 &\mbox{ and }&&
 \|\xi\|_\PhiSpaceL &:= \sqrt{\<\xi,\xi\>_\PhiSpaceL} \\
 \< \xi,\eta\>_\PhiSpace &:= \<\xi,\eta\>_\PhiSpaceL 
  + \<\partial_t\xi,\partial_t\eta\>_\PhiSpaceL
 &\mbox{ and }&&
 \|\xi\|_\PhiSpace &:= \sqrt{\<\xi,\xi\>_\PhiSpace}
\end{align*}
for $\xi,\eta\in \PhiSpace$. Here, $h_i$ denotes the initial point distance.
We set
\begin{align*}
  y_h(x,t) := \sum_{i=1}^N \alpha_i(t) \delta_\eps(x-\Phi_i(t))
\end{align*}
where the particle positions $\Phi_i$ are given by
\begin{align*}
  \partial_t \Phi(t) &= y_h(\Phi(t),t)
\end{align*}
for $\Phi\in P$. Then we get the particle representation
\begin{align}
 \label{equ:discrState}
 \begin{aligned}
 \<\partial_t y_h(t),\phi\>_{V^*,V} + \<y_h(t)\partial_x y_h(t),\phi\>_H
  + \nu\<\partial_x y_h(t),\partial_x \phi\>_H &= \<\calB u(t),\phi\>_H \\
 \<y(0),\psi\>_H &= \<y_0,\psi\>_H \\
 \end{aligned}
\end{align}
for all $\phi\in V$, $\psi\in H$.

\begin{remark}
 For the numerical implementation we use, similar to the finite element method, test functions of the form
 \begin{align*}
  \phi(x,t) := \sum_{i=1}^N a_i(t) \delta_\eps(x-\Phi_i(t))  
 \end{align*}
 which yield mass matrices.
 Moreover, we only discretise the support of the initial value (plus neighbourhood), i.e. if $\overline{\supp y_0}=[a,b]$ then for small $\tilde\eps>0$
 \[  a-\tilde\eps \le  X_i \le b+\tilde\eps\]
 holds for all $i=1,\ldots,N$.
\end{remark}

The optimisation is performed by a ``first optimise, then discretise'' approach \cite{JM_NM_RP_Basics}. Hence, we discretise the adjoint equation \eqref{equ:adjoint} separately. In order to avoid interpolations we choose the same point set as obtained by the forward system. In particular, we get for
\begin{align}
 \label{equ:phRepresent}
   p_h(x,t) := \sum_{i=1}^N \beta_i(t) \delta_\eps(x-\Phi_i(t))
\end{align}
the particle representation of the adjoint equation as
\begin{align}
 \label{equ:adjointEquationDiscr}
 \begin{aligned}
  \<\partial_t p_h(t),\phi\>_{V^*,V} + \<y_h(t)\partial_x p_h(t), \phi\>_H - \nu\<\partial_x p_h(t),\partial_x\phi\>_H &= 0 \\
  \< p_h(T), \psi\>_H &= \< y_d - y_h(T), \psi\>_H
 \end{aligned}
\end{align}
for all $\phi\in V$ and $\psi\in H$.
The discrete minimisation problem is then
\begin{flalign*}
 &\mbox{\minDiscr}\hspace*{30mm}
 \mbox{minimise }J(y_h,u)
 \qquad\mbox{subject to}\qquad
 \eqref{equ:discrState} &
\end{flalign*}

\noindent First, we show the existence of a unique discrete solution $(y_h,\Phi)$ to \eqref{equ:discrState} and its boundedness.

\begin{Assumption}
 \label{asp:existenceForwardDiscr}
 Let $y_0\in \compactVspace$ and $u\in U_{ad}$. Then the discrete problem \eqref{equ:discrState} has a unique solution $(y_h,\Phi)\in W(V)\cap L^\infty(V) \times P$.
\end{Assumption}

\begin{Assumption}
 Let $(y_h,\Phi,u)\in X_h$ be the solution to \eqref{equ:discrState}.
 Then the discrete adjoint equation \eqref{equ:adjointEquationDiscr} has a unique solution $p_h\in W(V)$.
\end{Assumption}

\begin{Theorem}
 \label{thm:boundedForwardDiscr}
 Let $y_0 \in \compactVspace$ and $u\in U_{ad}$. Furthermore, let $(y_h,\Phi)$ be as stated in theorem \ref{thm:interpolationError}.
 Then the discrete solution $(y_h,\Phi)\in W(V)\cap L^\infty(V) \times \PhiSpace$ satisfies
  \[ \|y_h\|_{L^\infty(V)} + \| y_h \|_{W(V)}  \le \Cstate   \]
 and
  \[ \|\Phi\|_\PhiSpace \le  \Cstate + C_{supp} \]
 holds for a constant $C_{supp}>0$ depending on $\compactVspace$ only and $\Cstate>0$ as defined in theorem \ref{thm:boundedForward}.
\end{Theorem}
\begin{proof}
 The estimation of $\|y_h\|_{L^\infty(V)} + \| y_h \|_{W(V)}$ is analogous to the proof of theorem \ref{thm:boundedForward}, in particular
 \begin{align*}
  \|y_h\|_{L^\infty(V)} + \| y_h \|_{W(V)} \le C( \|u\|_{L^2(0,T)} + \|y_0\|_V )^2 \exp( C(\|u\|_{L^2(0,T)}^2 + \|y_0\|_H^2) )
 \end{align*}
 For all $\xi\in \PhiSpaceL$ we have
 \begin{align*}
  \<\partial_t \Phi,\xi\>_\PhiSpaceL = \<y_h(\Phi,\cdot),\xi\>_\PhiSpaceL
 \end{align*}
 which yields for $\xi=\partial_t\Phi$
 \begin{align*}
  \|\partial_t\Phi\|_\PhiSpaceL^2 \le c\|y_h\|_{L^2(L^\infty)}\|\partial_t\Phi\|_\PhiSpaceL 
  \le c\|y_h\|_{W(V)}\|\partial_t\Phi\|_\PhiSpaceL
 \end{align*}
 and hence
 \begin{align*}
  \|\partial_t\Phi\|_\PhiSpaceL \le  c\|y_h\|_{W(V)}
  \le \Cstate
 \end{align*}
 Moreover,
 \begin{align*}
  \Phi_i(t) = \int_0^t y_h(\Phi_i(\tau),\tau) \,d\tau + \Phi_i(0)
  \;\Rightarrow\;
  \|\Phi\|_\PhiSpaceL \le C\|y_h\|_{L^2(L^\infty)} + C
 \end{align*}
 as $\Phi_i(0)<C$ depending on the support of the initial condition, cf. previous remark.
 
\end{proof}

\begin{Theorem}
 \label{thm:convergenceForward}
 Let $y\in W(H^m)\cap L^\infty(V)$, $m\ge 1$ be the solution to the continuous
 system and $y_h\in W(V)\cap L^\infty(V)$ the solution to the
 discrete system. Then
  \[ \| y - y_h \|_{W(V)} \le C (\errorH) \]
 holds.
\end{Theorem}
\begin{proof}
 The continuous solution $y\in W(V)$ satisfies
 \[
   \<\partial_t y(t),\phi\>_{V^*,V} +
   \<y(t)\partial_x y(t), \phi\>_H + \nu \<\partial_x y(t),\partial_x\phi\>_H 
    = \< \calB u(t), \phi\>_{V^*,V}
 \]
 for all $\phi\in V$ and the discrete solution $y_h\in W(V)$
 \[
   \<\partial_t y_h(t),\phi_h\>_{V^*,V} +
   \<y_h(t)\partial_x y_h(t), \phi_h\>_H 
   + \nu \<\partial_x y_h(t),\partial_x\phi_h\>_H 
    = \< \calB u(t), \phi_h\>_{V^*,V}
 \] 
 for all $\phi_h\in V$.
 Hence, by using the fact that
 \begin{align*}
  \<y(t)\partial_x y(t),\phi\>_H = -\<y(t)\partial_x y(t),\phi\>_H
  -\<y(t)^2,\partial_x\phi\>_H
 \end{align*}
 and with $z:=y-y_h$ we get
 \begin{align*}
   \<\partial_t z(t),\phi_h(t)\>_{V^*,V}
   -\frac12 \<(y(t)+y_h(t))z(t),\partial_x\phi_h(t)\>_H 
   + \nu \<\partial_x z(t),\partial_x\phi_h(t)\>_H 
    = 0
 \end{align*}
 We start with estimating $\|\partial_t z\|_{L^2(V^*)}$ by setting
 \[ \phi_h = 
   \underbrace{\Pi_\eps^h \RieszIsoAlt\partial_t y - \RieszIsoAlt\partial_t y}_{=:\phi_I}
   + \underbrace{\RieszIsoAlt\partial_t y
   - \RieszIsoAlt\partial_t y_h}_{=\RieszIsoAlt\partial_t z}
 \]
 \begin{align*}
  \|\partial_t z\|_{L^2(V^*)}^2 &=
   \<\partial_t z,\RieszIsoAlt\partial_t z\>_{L^2(V^*),L^2(V)}
   = -\<\partial_t z,\phi_I\>_{L^2(V^*),L^2(V)} 
    + \<\partial_t z,\phi_h\>_{L^2(V^*),L^2(V)}\\
  &= -\<\partial_t z,\phi_I\>_{L^2(V^*),L^2(V)} 
   - \nu \<\partial_x z,\partial_x \phi_h\>_{L^2(H)}
   +\frac12 \<(y + y_h)z , \partial_x\phi_h\>_{L^2(H)} \\
  &\le \|\partial_t z\|_{L^2(V^*)}\|\phi_I\|_{L^2(V)}
   + C \|z\|_{L^2(V)}\|\phi_h\|_{L^2(V)}
   + C\|y+y_h\|_{L^2(L^\infty)}\|z\|_{L^2(H)}\|\phi_h\|_{L^2(V)} \\
  &\le \|\partial_t z\|_{L^2(V^*)}\|\phi_I\|_{L^2(V)}
   + C\|z\|_{L^2(V)}(\|\phi_I\|_V + \|\partial_t z\|_{L^2(V^*)})
 \end{align*}
 as $\|\phi_h(t)\|_V\le \|\phi_I\|_V+\|\partial_t z\|_{V^*}$ and $V\hookrightarrow H$.
 Using Young's inequality gives
 \begin{align*}
  \|\partial_t z\|_{L^2(V^*)}^2 &\le \gamma_1\|\partial_t z\|_{L^2(V^*)}^2
   + C_{\gamma_1}\|\phi_I\|_{L^2(V)}^2
   + C(\|z\|_{L^2(V)}^2 + \|\phi_I\|_{L^2(V)}^2) \\
   &\quad+ C_{\gamma_2}\|z\|_{L^2(V)}^2 + C \gamma_2 \|\partial_t z\|_{L^2(V^*)}^2
 \end{align*}
 an by choosing $\gamma_1 = \frac14$ and $\gamma_2 = \frac{1}{4C}$ we obtain
 \begin{align*}
  \frac12 \|\partial_t z\|_{L^2(V^*)}^2 
  &\le C( \|\phi_I\|_{L^2(V)}^2 + \|z\|_{L^2(V)}^2 )
 \end{align*}
 which yields
 \begin{align}
  \label{equ:y-y_hVstar}
  \|\partial_t z\|_{L^2(V^*)} \le C(\|\phi_I\|_{L^2(V)} + \|z\|_{L^2(V)})
 \end{align}
 To estimate the $L^2(V)$ error of $y-y_h$ we set $\phi_h=\phi_I+z$ where 
 $\phi_I = \Pi^h_\eps y - y$. 
 \begin{align*}
  \frac{1}{2}\frac{d}{dt}\|z(t)&\|_H^2 
   + \nu\|z(t)\|_V^2 
   =
   \nu\|z(t)\|_H^2 +
   \frac12 \<(y(t)+y_h(t))z(t), \phi_I(t)+z(t)\>_H \\
   &\hspace{30mm}- \nu \<\partial_xz(t),\partial_x\phi_I(t)\>_H
   - \<\partial_t z(t),\phi_I(t)\>_{V^*,V}  \\
  &\le \nu\|z(t)\|_H^2
   + C\|y(t)+y_h(t)\|_{L^\infty}\|z(t)\|_H\|\phi_I(t)+z(t)\|_H
   + C \|z(t)\|_V\|\phi_I(t)\|_V \\
   &\quad + \|\partial_t z(t)\|_{V^*}\|\phi_I(t)\|_V \\
  &\le (\nu+C)\|z(t)\|_H^2  + C\|z(t)\|_H\|\phi_I(t)\|_H
   + C \|z(t)\|_V\|\phi_I(t)\|_V
   + \|\partial_t z(t)\|_{V^*}\|\phi_I(t)\|_V
  \intertext{as $\|y_h(t)+y(t)\|_{L^\infty}\le C$ and due to Young's inequality}
  &\le C\|z(t)\|_H^2 + C\|\phi_I(t)\|_V^2 
   + C\gamma_1\|z(t)\|_V^2 + C_{\gamma_1} \|\phi_I(t)\|_V^2 
   + \gamma_2\|\partial_t z(t)\|_{V^*}^2 + C_{\gamma_2}	\|\phi_I(t)\|_V^2
 \end{align*}
 holds.
 Setting $\gamma_1 = \frac{\nu}{2C}$ and integrating over $(0,T)$ yields
 \begin{align*}
  \|z(T)\|_H^2 - \|z(0)\|_H^2 + \nu \|z\|_{L^2(V)}^2
  \le C_{\gamma_2}\|\phi_I\|_{L^2(V)}^2 + \gamma_2 \|\partial_t z\|_{L^2(V^*)}^2 
  + C\int_0^T \|z(t)\|_H^2\,dt
 \end{align*}
 Using \eqref{equ:y-y_hVstar} we obtain
 \begin{align*}
  \|z(T)\|_H^2 - \|z(0)\|_H^2 + \nu \|z\|_{L^2(V)}^2
  \le C_{\gamma_2}\|\phi_I\|_{L^2(V)}^2 + \gamma_2 C \|z\|_{L^2(V)}^2 
  + C\int_0^T \|z(t)\|_H^2\,dt
 \end{align*}
 and therefore
 \begin{align*}
  \|z(T)\|_H^2 +  \frac\nu2 \|z\|_{L^2(V)}^2
  \le C\|\phi_I\|_{L^2(V)}^2 + \|z(0)\|_H^2 + C\int_0^T \|z(t)\|_H^2\,dt
 \end{align*} 
 by setting $\gamma_2 = \frac{\nu}{2C}$.
 Due to Gronwall's lemma we get
 \begin{align*}
  \|z\|_{L^2(V)}^2 \le C \big(\|\phi_I\|_{L^2(V)}^2 + \|z(0)\|_H^2 \big)
 \end{align*}
 and hence
 \begin{align*}
  \|z\|_{L^2(V)} \le C \big(\|\phi_I\|_{L^2(V)} + \|z(0)\|_H \big).
 \end{align*} 
 Now we consider the right hand side terms.
 \begin{align*}
  \|\phi_I\|_{L^2(V)} \le C(\errorH)
 \end{align*}
 due to theorem \ref{thm:interpolationError}.
 The second term is given by
 \begin{align*}
  \|z(0)\|_H = \|y(0)-y_h(0)\|_H = \|y(0)-y_I(0)\|_H = \|\phi_I(0)\|_H \le C(\errorL).
 \end{align*}
 due to theorem \ref{thm:interpolationError} again and the fact that the discrete initial value is defined by $y_h(0) = \Pi_\eps^h y_0$.
 
 Combining all results finally yields the assumption.
 
\end{proof}

\begin{Theorem}
 There exists an optimal solution to \minDiscr.
\end{Theorem}
\begin{proof}
 Due to theorem \ref{asp:existenceForwardDiscr} there exists a unique
 $(y_h,\Phi)\in W(V)\cap L^\infty(V)\times \PhiSpace$ for every $u\in U_{ad}$.
 Moreover, $J(y_h,u)\ge 0$ for all $y_h\in W(V)$ and $u\in U_{ad}$.
 Hence,
 \begin{align*}
  j = \inf_{u\in U_{ad}} J(y_h(u),u) \in\mathbb R^+_0
 \end{align*}
 exists. We define the minimising sequence 
 $({y_h}_n,\Phi_n,u_n)\in W(V)\times P\times U_{ad}$ by
 \begin{align*}
  J({y_h}_n,u_n)\to j \quad\mbox{ as }\quad n \to \infty
 \end{align*}
 and $({y_h}_n,\Phi_n,u_n)$ solves \eqref{equ:discrState}.
 
 The convergence for the $y_h$ equations are analogous to the proof of theorem \ref{thm:existOptimalSolution}
 As $\Phi$ is bounded in $\PhiSpace$ there exists a subsequence with
 \begin{align*}
  \Phi_n \rightharpoonup \bar\Phi ,\quad n\to\infty\qquad \mbox{in }\PhiSpace
 \end{align*}
 and
 \begin{align*}
   \partial_t\Phi_n \rightharpoonup \partial_t\bar\Phi 
   ,\quad n\to\infty\qquad \mbox{in }\PhiSpaceL
 \end{align*}
 Since the embedding $H^1(0,T)\hookrightarrow L^2(0,T)$ is compact we get
 \begin{align}
  \label{equ:existDiscrOpt6}
  \Phi_n \to \bar\Phi ,\quad n\to\infty\qquad \mbox{in }\PhiSpaceL
 \end{align}
 Moreover, we have for $\xi\in \PhiSpaceL$
 \begin{align*}
  \<{y_h}_n(\Phi_n,\cdot)-\bar{y_h}(\bar\Phi,\cdot),\xi\>_{\PhiSpaceL}
  &=  \<{y_h}_n(\Phi_n,\cdot)-\bar{y_h}(\Phi_n,\cdot),\xi\>_{\PhiSpaceL}
   + \<\bar{y_h}(\Phi_n,\cdot)-\bar{y_h}(\bar \Phi,\cdot),\xi\>_{\PhiSpaceL} \\
  &\le \<{y_h}_n(\Phi_n,\cdot)-\bar{y_h}(\Phi_n,\cdot),\xi\>_{\PhiSpaceL}
   + \|\bar{y_h}(\Phi_n,\cdot)-\bar{y_h}(\bar\Phi,\cdot)\|_{\PhiSpaceL}\|\xi\|_\PhiSpaceL
\\
  &\le \<{y_h}_n(\Phi_n,\cdot)-\bar{y_h}(\Phi_n,\cdot),\xi\>_{\PhiSpaceL}
    + C\|\Phi_n-\Phi\|_{\PhiSpaceL}^\alpha\|\xi\|_\PhiSpaceL
 \end{align*}
 with $0<\alpha<\frac12$ as $y\in W(V)\cap L^\infty(V)$ and hence 
 $y(t)\in C^{0,\alpha}(\mathbb R)$ as the embedding $V\hookrightarrow C^{0,\alpha}(\mathbb R)$ holds
 due to Morrey's lemma, cf. e.g. \cite{Alt}.
 Due to $y_h(t)\rightharpoonup^* {\bar y_h(t)}$ in $L^\infty(\mathbb R)$ and \eqref{equ:existDiscrOpt6} we obtain
 \begin{align*}
  \<{y_h}_n(\Phi_n,\cdot)-\bar{y_h}(\bar\Phi,\cdot),\xi\>_{\PhiSpaceL} \to 0
 \end{align*}
 which yields
 \begin{align*}
  \<\partial_t {\Phi}_n,\xi\>_{\PhiSpaceL} -
  \<{y_h}_n({\Phi}_n,\cdot),\xi\>_{\PhiSpaceL}
  \to \<\partial_t {\bar\Phi},\xi\>_{\PhiSpaceL}
  - \<\bar{y_h}(\Phi,\cdot),\xi\>_{\PhiSpaceL}
 \end{align*}
 for all $\xi\in \PhiSpaceL$.
 
 Combining all above results gives $(\bar y_h,\bar\Phi,\bar u)$ also solves \eqref{equ:discrState}.
 Due to definition $J$ is lower-semi-continuous, i.e.
 \begin{align*}
  J({\bar y_h},\bar u) \le \liminf_{n\to\infty} J({y_h}_n,u_n) = j
 \end{align*}
 As $J({\bar y_h},\bar u)\le j$ we obtain $J({\bar y_h},\bar u) = j$ and hence $j$ is a
 minimum.
 
\end{proof}

\begin{Theorem}
 \label{thm:boundedAdjointDiscr}
 Let $(y_h,\Phi,u)\in X_h$ be the solution to \eqref{equ:discrState} and
 $p_h\in W(V)$ the solution to \eqref{equ:adjointEquationDiscr}.
 Then the estimate
 \begin{align*}
  \|p_h\|_{W(V)} \le c(\Cstate^4+\Cstate^2) \|y_h(T)-y_d\|_H
 \end{align*}
 holds.
\end{Theorem}
\begin{proof}
 Since the supporting points are given, the proof is analogous to the theorem \ref{thm:boundedAdjointCont}.
 
\end{proof}

Now we state the difference of the adjoint equations derived in section
\ref{sec:contOptimisation} with the one obtained above, in particular
\eqref{equ:adjoint} and \eqref{equ:adjointEquationDiscr}, for a fixed $y\in W(V)\cap L^\infty(V)$.

\begin{Theorem}
 \label{thm:convergenceAdjoints}
 Let $y\in W(V)\cap L^\infty(V)$ be a given function with $\|y\|_{L^\infty(L^\infty)}\le C$.
 Let $p\in W(H^m)$, $m\ge 1$, be the solution to \eqref{equ:adjoint} and $p_h\in W(V)$ the solution to
 \eqref{equ:adjointEquationDiscr} with respect to $y$. Then the estimate
 \begin{align*}
  \|p-p_h\|_{W(V)} \le C(\errorH)
 \end{align*}
 holds.
\end{Theorem}
\begin{proof}
 Analogous to the proof of theorem \ref{thm:convergenceForward}.
\end{proof}

Finally, we show the convergence of the optimal control function obtained by the analytical
optimisation \minCont, denoted by $u\in U_{ad}$ and the numerical one \minDiscr, denoted by $u_h\in U_{ad}$.
\begin{Lemma}
 Let $u\in U_{ad}$ be a solution to \minCont and
 $u_h\in U_{ad}$ be a solution to \minDiscr. Then there exists $C>0$ independent of $h$ such that
 \begin{align*}
  \|u\|_{H^1(0,T)} \le C
  \qquad\mbox{ and }\qquad
  \|u_h\|_{H^1(0,T)} \le C
 \end{align*}
 holds.
\end{Lemma}
\begin{proof}
 We only show $\|u_h\|_{H^1(0,T)} \le C$. The first order optimality condition yields
 \begin{align*}
  s := (\RieszIsoAlt) \frac1\sigma \calB^* p_h
  \qquad\mbox{and}\qquad
  u_h = \projection(s)
 \end{align*}
 for a minimum $(y_h,\Phi,p_h,u_h)$. Hence we get
 \begin{align*}
  \| s \|_{H^1(0,T)}^2 &= \<s,s\>_{H^1(0,T)}
   =  \< \RieszIsoAlt \tfrac1\sigma \calB^* p_h, s\>_{H^1(0,T)}
   = \frac1\sigma \< \calB^* p_h, s\>_{H^1(0,T)^*,H^1(0,T)} \\
  &= \frac1\sigma \<\calB s, p_h\>_{L^2(V^*),L^2(V)}
   \le C \|s\|_{L^2(0,T)} \|p_h\|_{L^2(V)} \\
  &\le C \|s\|_{H^1(0,T)} \Cadjoint \|y_h(T)-y_d\|_H\\
  &\le C \|s\|_{H^1(0,T)} \Cadjoint ( \|y_h(T)\|_H + \|y_d\|_H )
 \end{align*}
 due to theorem \ref{thm:boundedAdjointDiscr} and the embedding $H^1(0,T)\hookrightarrow L^2(0,T)$.
 As a consequence of theorem \ref{thm:boundedForwardDiscr} we obtain
 \begin{align*}
  \| s \|_{H^1(0,T)} \le C \Cadjoint ( \Cstate + \|y_d\|_H )
 \end{align*}
 Note that
 \begin{align*}
  \Cstate = C(\nu, T, \chi_c) ( \|u_h\|_{L^2(0,T)} + \|y_0\|_V ) \exp( C(\|u_h\|_{L^2(0,T)}^2 + \|y_0\|_H^2) )
 \end{align*}
 is independent of $h$ and since $u_h=\projection(s)$ is bounded, i.e. $\|u_h\|_{L^\infty(0,T)}\le C$, we obtain
 \begin{align*}
  \|s\|_{H^1(0,T)} \le C
 \end{align*}
 Using the fact that $\projection$ is continuous, we obtain
 \begin{align*}
  \|u_h\|_{H^1(0,T)}  = \|\projection(s)\|_{H^1(0,T)} \le c\|s\|_{H^1(0,T)} \le C
 \end{align*}
 which proves the assumption. 
\end{proof}

\begin{Lemma}
 \label{lem:strongConvergence}
 Let $u\in U_{ad}$ be a solution to \minCont and $u_h\in U_{ad}$ to \minDiscr. Then there exists a subsequence $u_{h_n}$ of $u_h$ such that
 \begin{align*}
  u_{h_n}\to u \qquad\mbox{in }H^1(0,T) \qquad\mbox{as } n\to \infty
 \end{align*}
 holds.
\end{Lemma}
\begin{proof}
 We define the sequence $(y_h,p_h,u_h)\in W(V)\times W(V) \times U_{ad}$ such that it is
 a solution to \minDiscr.
 Since $u_h\in U_{ad}$ is bounded in $H^1(0,T)$ there exists a converging subsequence
 \begin{align*}
  u_h \rightharpoonup \tilde u,
  \quad h\to 0
  \qquad \mbox{in } H^1(0,T)
 \end{align*}
 As the embedding $H^1(0,T)\hookrightarrow L^2(0,T)$ is compact, we obtain a subsequence
 \begin{align}
  \label{equ:uhtouL2}
  u_h \to \tilde u,
  \quad h\to 0
  \qquad \mbox{in } L^2(0,T)
 \end{align}
 Since $U_{ad}$ is weakly closed, also $\tilde u\in U_{ad}$.
 Due to theorem \ref{thm:convergenceForward} we get
 \begin{align*}
  \|y(u)-y_h(u)\|_{W(V)} \to 0 \qquad\mbox{as } h\to 0
 \end{align*}
 where $y(u)$ denotes the solution of $e(y,u)=0$ and $y_h(u)$ to $e_h(y_h,\Phi_h,u)=0$.
 Hence, we obtain
 \begin{align*}
  \|y(\tilde u) - y_h(u_h)\|_{W(V)} &\le 
   \underbrace{\|y(\tilde u)-y_h(\tilde u)\|_{W(V)}}_{\to 0}
   + \underbrace{\|y_h(\tilde u)-y_h(u_h)\|_{W(V)}}_{\to 0}
 \end{align*}
 due to the continuity of the solution operator $y_h(u)$ and \eqref{equ:uhtouL2}.
 The same holds true for $p$, i.e.
 \begin{align*}
  \|p(y)-p_h(y)\|_{W(V)} \to 0 \qquad\mbox{as } h\to 0
 \end{align*} 
 Hence,
 \begin{align*}
  \|p(y) - p_h(y_h)\|_{L^2(V)} &\le 
   \underbrace{\|p(y)-p(y_h)\|_{L^2(V)}}_{\to 0}
   + \underbrace{\|p(y_h)-p_h(y_h)\|_{L^2(V)}}_{\to 0}
 \end{align*} 
 due to lemma \ref{lem:lipschitzAdjoint}.
 
 \noindent Combining the above results yields
 \begin{align*}
  \|p(y(\tilde u)) - p_h(y_h(u_h) \|_{L^2(V)} \to 0 \qquad\mbox{as }h\to 0
 \end{align*}
 We define
 \begin{align*}
  s_h &:= (\RieszIsoAlt) \tfrac1\sigma \calB^* p_h(y_h(u_h))
  \hspace{20mm} &u_h &= \projection(s_h) \\
  \tilde s &:= (\RieszIsoAlt) \tfrac1\sigma \calB^* p(y(\tilde u))
  &\tilde u &= \projection(\tilde s)
 \end{align*}
 we obtain
 \begin{align*}
  \|\tilde s - s_h\|_{H^1(0,T)}^2 
   &= \frac1\sigma \< (\RieszIsoAlt) \calB^*\big( p(y(\tilde u))
    -p_h(y_h(u_h))\big),\tilde s-s_h\>_{H^1(0,T)} \\
  &= \frac1\sigma \< \calB^*(p(y(\tilde u))
   -p_h(y_h(u_h))), \tilde s-s_h\>_{H^1(0,T)^*,H^1(0,T)} \\
  &= \frac1\sigma \< \calB(\tilde s-s_h), p(y(\tilde u))
   -p_h(y_h(u_h))\>_{L^2(V^*),L^2(V)} \\
  &\le C \|p(y(\tilde u))-p_h(y_h(u_h))\|_{L^2(V)} \| \tilde s-s_h\|_{L^2(0,T)}
 \end{align*}
 and therefore
 \begin{align*}
  \|\tilde s - s_h\|_{H^1(0,T)} \le C \|p(y(\tilde u))-p_h(y_h(u_h))\|_{L^2(V)}
  \to 0\qquad\mbox{as }h\to 0
 \end{align*}
 Since the projection $\projection$ is Lipschitz continuous, we get
 \begin{align*}
  \|\tilde u-u_h\|_{H^1(0,T)} \le C\|\tilde s-s_h\|_{H^1(0,T)}
 \end{align*}
 and hence $\tilde u = u$ and we get for the defined subsequence
 \begin{align*}
  u_h \to  u \qquad\mbox{in }H^1(0,T)
 \end{align*}
 
\end{proof}

\begin{remark}
 It is possible to show that, satisfying the second order sufficiency condition,
 there exists a $h_0>0$ such that for all  $h<h_0$
   \[ \|u_h-u\|_{H^1(0,T)} \le C\big( \errorH \big) \]
 holds.
 For further details we refer to \cite{MyPhd}.
\end{remark}

\section{Numerical Results}
In this section we verify the derived convergence rates numerically. 
The setting used is $\nu=1, T=1, u_l=0, u_u=100$ and $y(0)\equiv 0$. The cost functional is 
\[ J(y,u) = \frac12 \|y(1)-10\exp(-2x^2)\|_H^2 + \frac{0.05}{2} \|u\|_{H^1(0,1)}^2. \]
and the localisation function $\chi_c(x) = \exp(-5x^2)$.
Moreover, the number of time steps is set to $N_t=500$, which is large enough to avoid significant errors due to time integration.
To get a reference solution, we perform the optimisation of the viscous Burger's equation on a fine fixed grid ($h\simeq 10^{-3}$) and denote the solution by $(y,p,u)$ in the following. Then the particle solutions are evaluated by using a steepest descent algorithm with Armijo rule, see e.g. \cite{Kelly}.
The $\delta_\eps$-function for the interpolation operator $\Pi_\eps^h$ is given by
\begin{align*}
 \delta_\eps(x) := \frac{1}{\sqrt{\pi}\eps} \exp(-\eps^{-2}x^2)
\end{align*}
which satisfies $r=2$ for the momentum stated in lemma \ref{lem:convolutionEstimate}.

%

We start with an optimisation for $\eps=0.3$ and $h=0.1$. The results, see figure~\ref{fig:result1}, show a good convergence rate. As expected, we have a steeper decrease of the gradient norm during the first steps, then we get a more or less stable convergence rate of approximately $1.0$. No Armijo step size reductions are needed. The expected final state is qualitatively reached but, due to the regularisation, we only reach a maximal value of $~5$ instead of the expected value $10$.

Then we verify the convergence rate of the continuous optimisation, in particular the reference solution, and the discrete one. For this we first fix the point distance $h$ to $h=0.1$ and vary $\eps$. Hence, we expect an error of
\begin{align*}
 \| u_h-u\|_{H^1(0,T)} , \| y_h-y \|_{L^2(V)} \propto \eps^2 + h^m \eps^{-3}
\end{align*}
where $m$ depends on the continuous solution.
The $H^1(0,T)$ error for $u$ and the $L^2(V)$ error for $y$, see figure~\ref{fig:results2a}, show a good coincidence with the predicted error.

Next, we fix $\eps$ to $\eps=0.1$ and vary the initial point distance $h$. Figure~\ref{fig:results2b} shows again the $H^1(0,T)$ error for $u$ and the $L^2(V)$ error for $y$.
The expected relation
\begin{align*}
 \| u_h-u\|_{H^1(0,T)} , \| y_h-y \|_{L^2(V)} \propto C + h^m
\end{align*}
is satisfied for the $u$-difference quickly, i.e. from $1/h=50$ on we obtain a constant error norm.
The difference of $y$ has more regularity, i.e the $h^m$ term is dominating in contrast to the constant $C$.

Finally, we choose $\eps$ depending on $h$ by the $\eps = h^\frac{1}{2}$. Hence, we expect for $m=2$
\begin{align*}
 \| u_h-u\|_{H^1(0,T)} , \| y_h-y \|_{L^2(V)} \propto h^\frac{1}{2}
\end{align*}
which is validated by figure~\ref{fig:results3}.

\section{Conclusion}
In this paper we derived the optimality system for the minimisation of a cost functional subject to the viscous Burgers equation in the whole space $\mathbb R$. First, the boundedness of the state and adjoint equation was stated. Moreover, we have shown the existence of a minimiser to the continuous optimisation problem and the existence and uniqueness of a solution to the adjoint system. Both systems, the state and adjoint system, were discretised by a particle approximation obtained by Dirac sequence and the corresponding discretisation errors were derived. Further, we proved the existence of a strong converging subsequence of discrete optimal control functions to the continuous optimal control.
Finally, the derived convergence rates are verified numerically.

\bibliography{Literature}

\begin{figure}[p]
 \hspace{-1cm}
 \subfigure[Optimal state $y$.]{
  \includegraphics[scale=0.7]{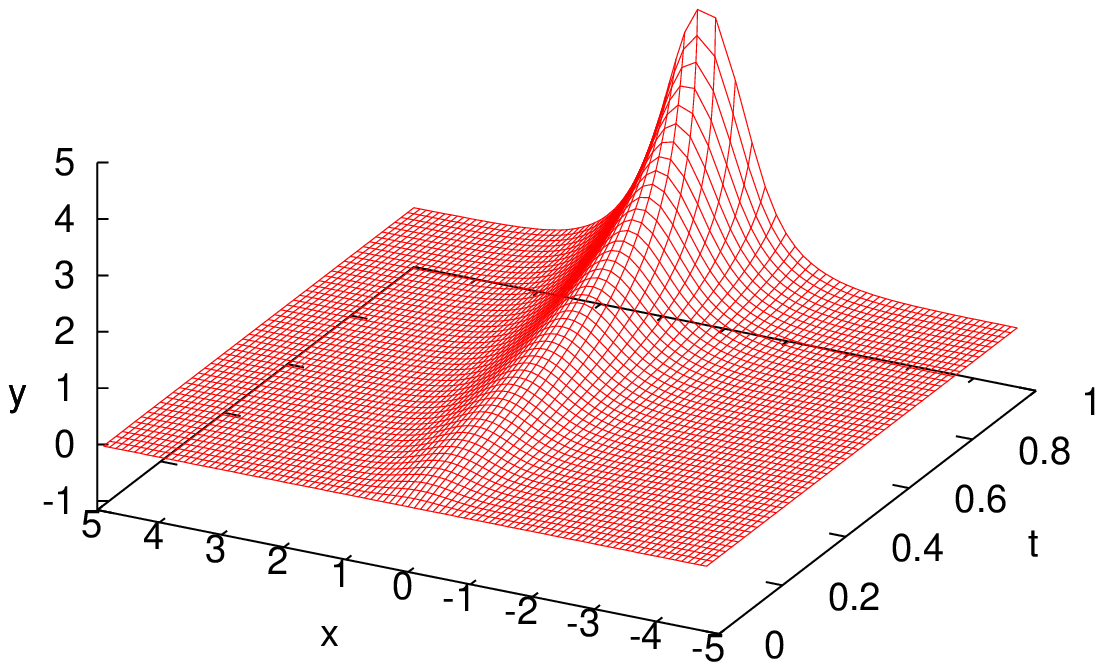}
 }
 \hspace{-5mm}
 \subfigure[Optimal control $u$.]{
  \includegraphics[scale=0.55]{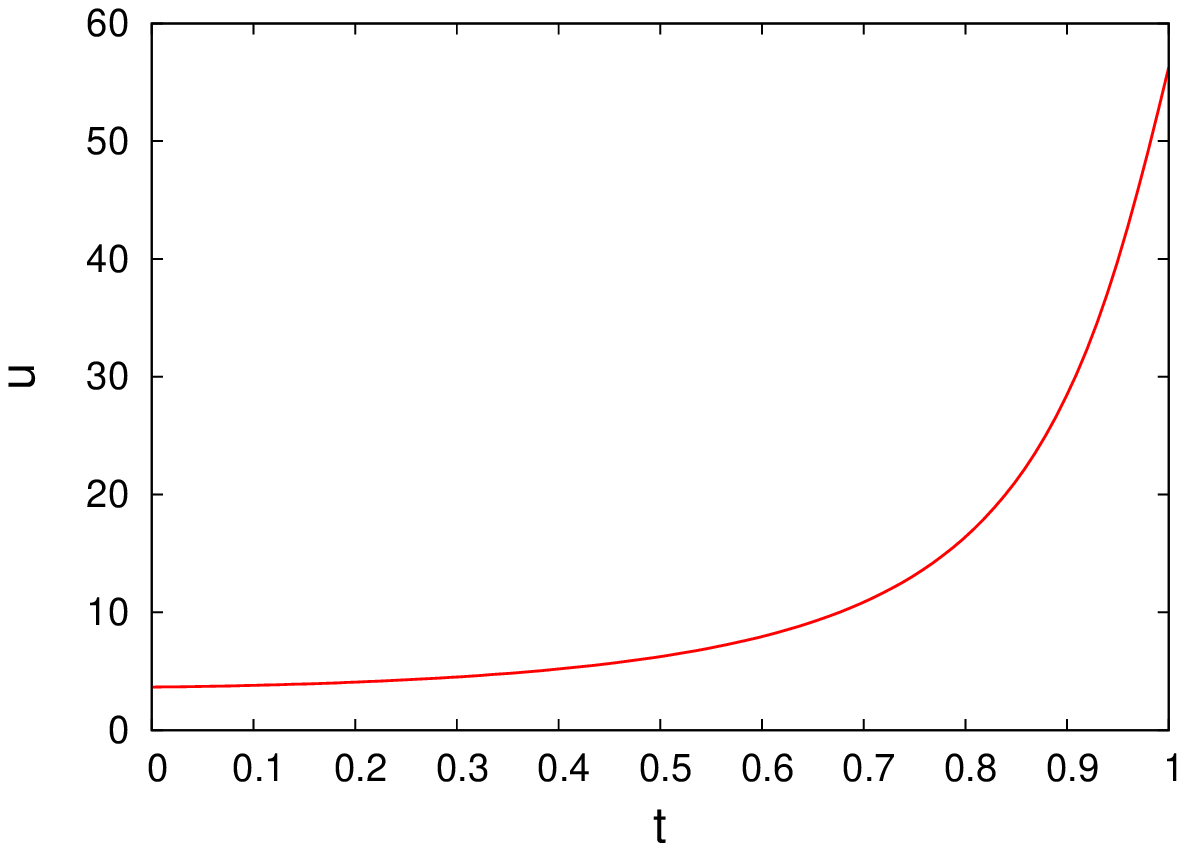}
 }
 \subfigure[Cost function.]{
  \includegraphics[scale=0.55]{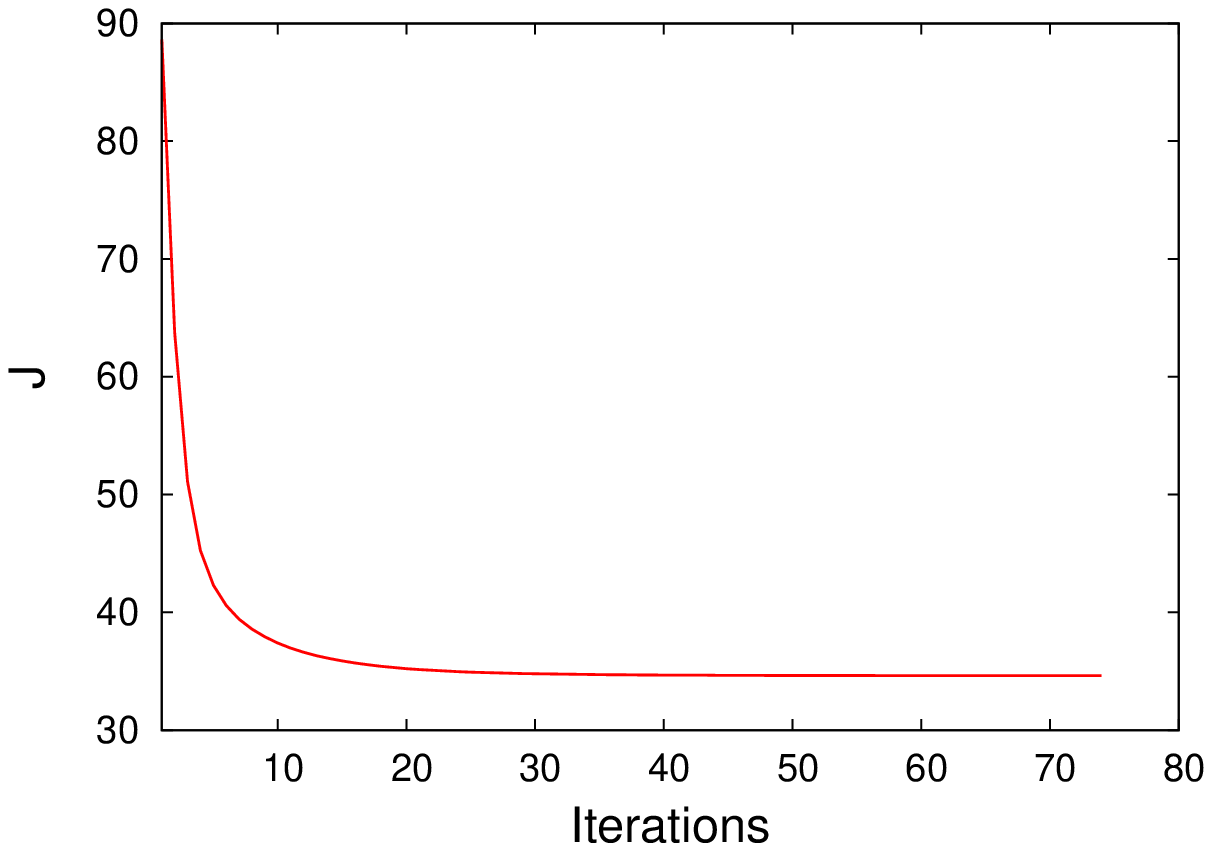}
 }
 \hspace{5mm}
 \subfigure[Relative gradient norm ($H^1(0,T)$-norm).]{
  \includegraphics[scale=0.55]{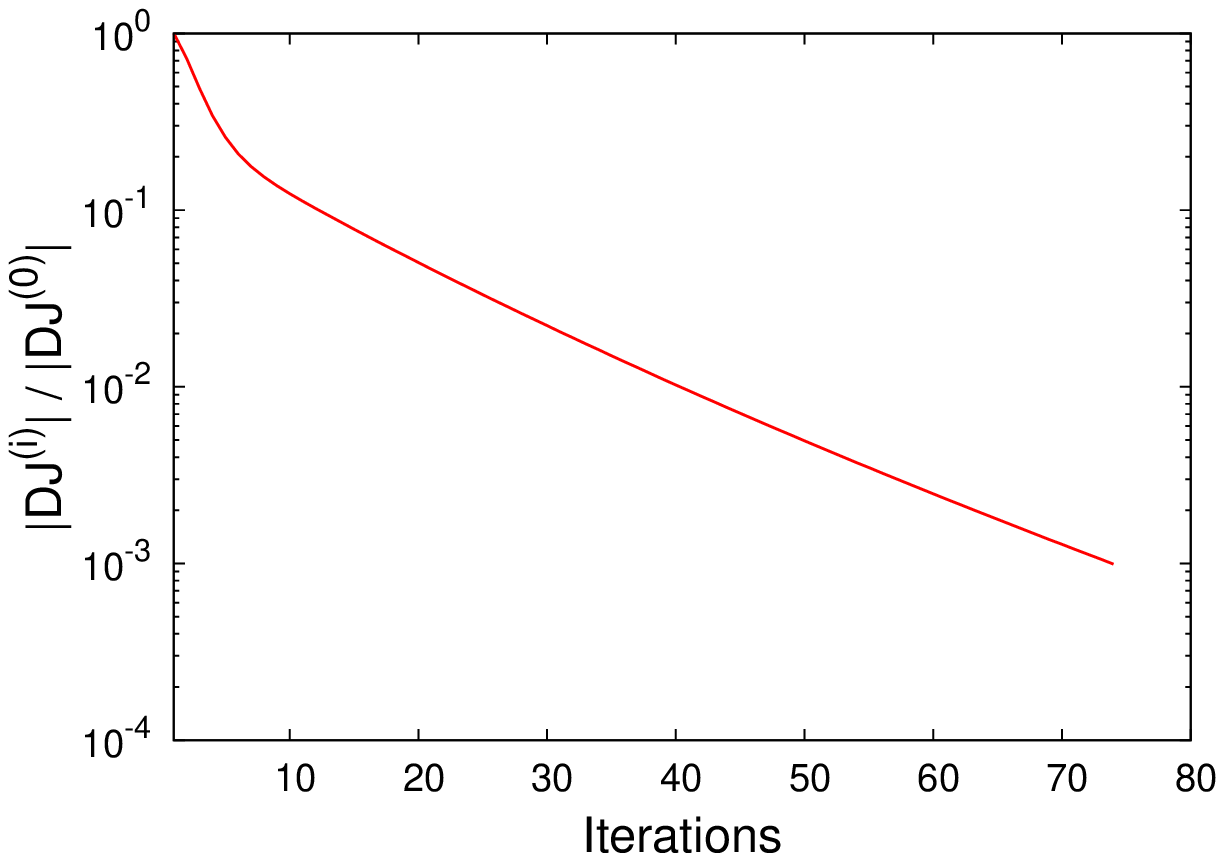}
 }
 \caption{Optimal particle solution.}
 \label{fig:result1}
\end{figure}

\begin{figure}[p]
 \subfigure[\label{fig:results2a}Varying $\epsilon$ for fixed $h=0.1$.]{
  \includegraphics[scale=0.55]{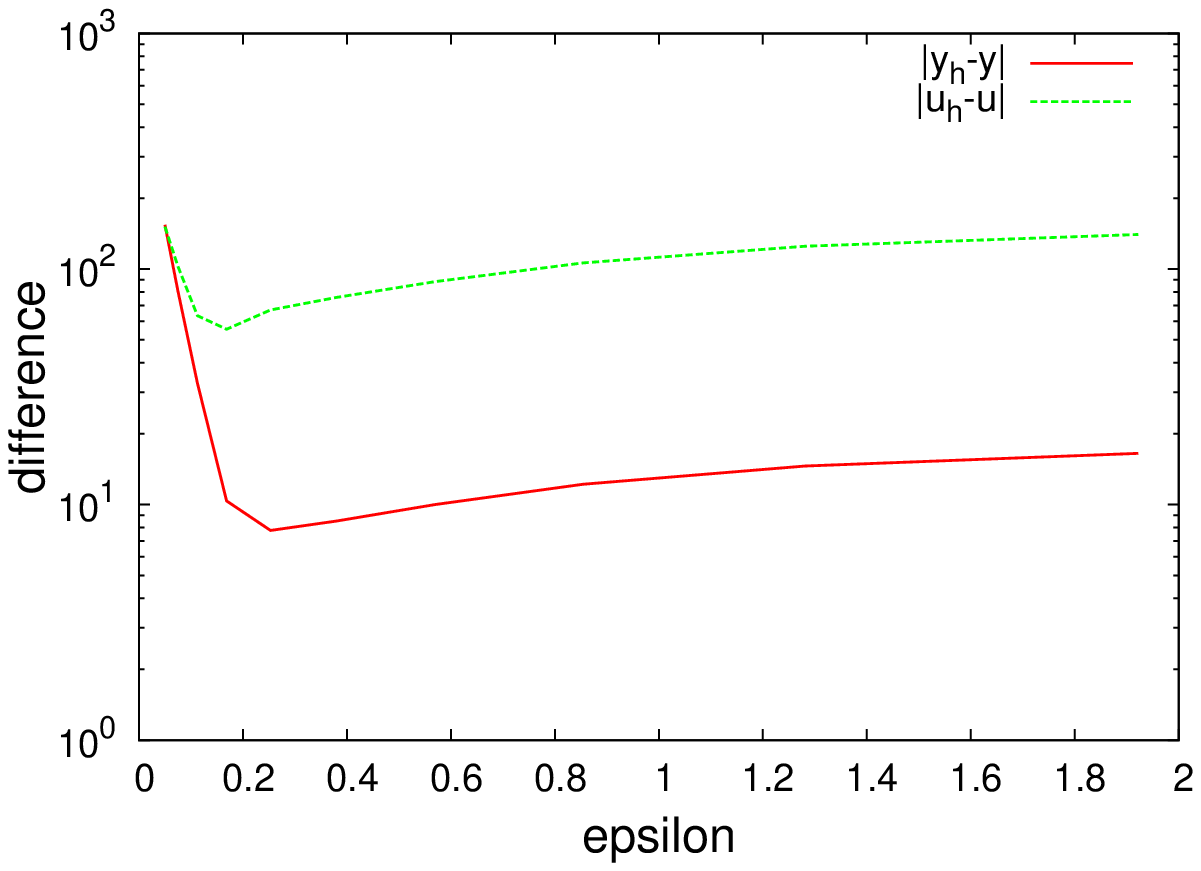}
 } 
 \subfigure[\label{fig:results2b}Varying $h$ for fixed $\eps=0.1$.]{
  \includegraphics[scale=0.55]{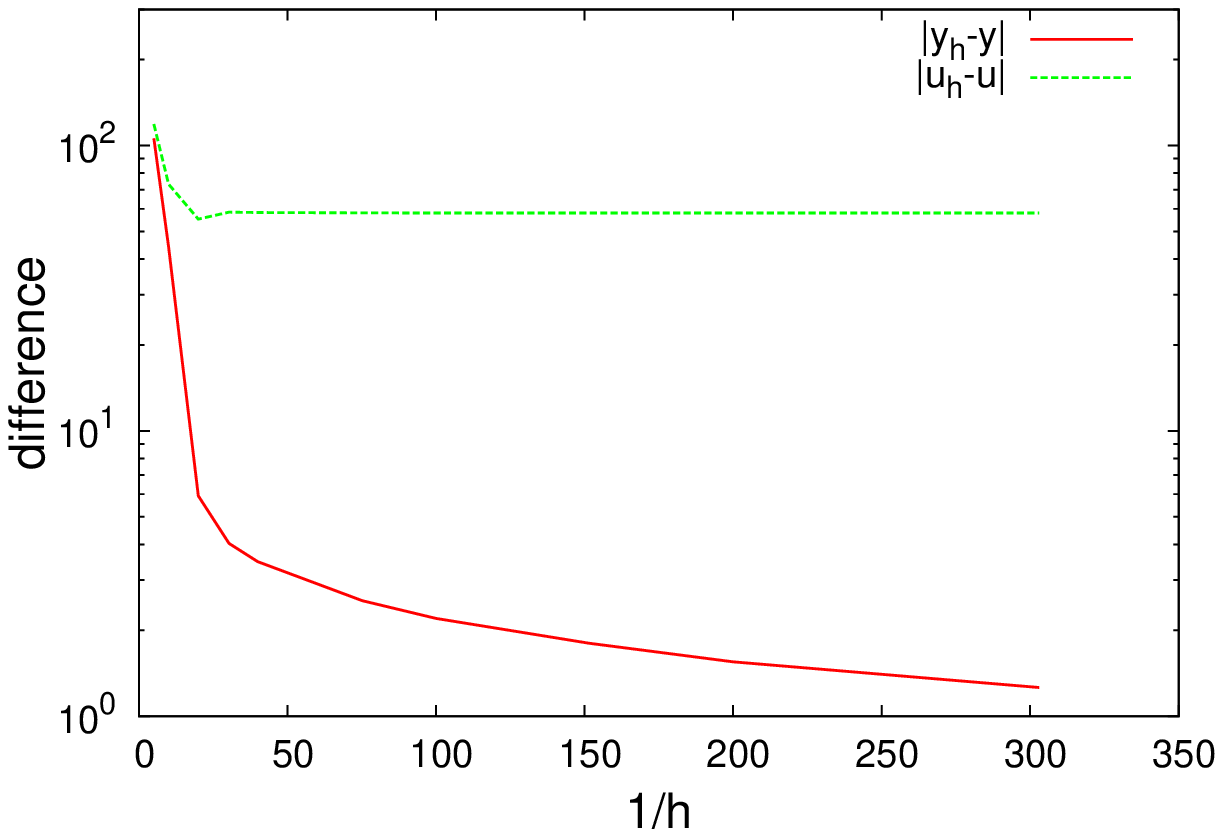}
 }
 \caption{Convergence ($\|y_h-y\|_{L^2(V)}$ and $\|u_h-u\|_{H^1(0,T)}$) for varying $\epsilon$ and $h$.}
\end{figure}

\begin{figure}
 \centering
 \includegraphics[scale=0.55]{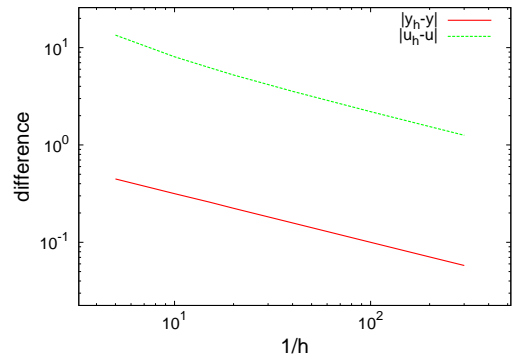}
 \caption{Convergence ($\|y_h-y\|_{L^2(V)}$ and $\|u_h-u\|_{H^1(0,T)}$) for $\eps=h^\frac{1}{2}$.}
 \label{fig:results3}
\end{figure}

\end{document}